\documentclass{amsart}
\usepackage[utf8]{inputenc}
\usepackage{amsmath}
\usepackage{amssymb}
\usepackage{mathrsfs}
\usepackage{graphicx}

\usepackage[colorlinks]{hyperref}
\usepackage[svgnames, dvipsnames, usenames]{xcolor}
\hypersetup{urlcolor = RedViolet, linkcolor = RoyalBlue, citecolor = ForestGreen}

\usepackage{tikz}
\usepackage{partmac}

\def\ppen{\penalty 300 }
\let\col=\colon
\def\colon{\col\ppen}

\theoremstyle{plain} 
\newtheorem{thm}{Theorem}[section]
\newtheorem*{thm*}{Theorem}

\newtheorem{lem}[thm]{Lemma}

\theoremstyle{definition}
\newtheorem{defn}[thm]{Definition}
\newtheorem{rem}[thm]{Remark}
\newtheorem{ex}[thm]{Example}

\numberwithin{equation}{section}

\renewcommand{\theta}{\vartheta}
\renewcommand{\phi}{\varphi}
\renewcommand{\epsilon}{\varepsilon}
\renewcommand{\subset}{\subseteq}

\newcommand{\N}{\mathbb N}
\newcommand{\Z}{\mathbb Z}

\newcommand{\R}{\mathbb R}
\newcommand{\C}{\mathbb C}

\newcommand{\staralg}{\mathop{\rm\ast\mathchar `\-alg}}
\DeclareMathOperator{\Mor}{Mor}
\DeclareMathOperator{\id}{id}

\DeclareMathOperator{\spanlin}{span}

\newcommand{\RCat}{\mathfrak{C}}

\newcommand{\Lin}{\mathscr{L}}

\newcommand{\Oalg}{\mathscr{O}}

\newcommand{\forkpart}{\Partition{
\Pblock 0to0.5:1,2
\Pline (1.5,0.5) (1.5,1)
}}

\newcommand{\immersepart}{\Partition{
\Pline (2,1) (1,0)
\Pline (3,1) (4,0)
\Pblock 0to0.3:2,3
}}

\newcommand{\IcrosspartI}{\Partition{
\Pline (1,1) (1,0)
\Pline (2,1) (3,0)
\Pline (3,1) (2,0)
\Pline (4,1) (4,0)
}}

\begin{document}
\title{Presentations of projective quantum groups}
\author{Daniel Gromada}
\address{Czech Technical University in Prague, Faculty of Electrical Engineering, Department of Mathematics, Technická 2, 166 27 Praha 6, Czechia}
\email{gromadan@fel.cvut.cz}
%\thanks{}
\thanks{This work was supported by the project OPVVV CAAS CZ.02.1.01/0.0/0.0/16\_019/0000778}
\thanks{I would like to thank Makoto Yamashita and the referee for providing references to Remark~\ref{R.moneq} and Moritz Weber for pointing out a certain mistake in an earlier version of the manuscript.}
\date{\today}

\begin{abstract}
Given an orthogonal compact matrix quantum group defined by intertwiner relations, we characterize by relations its projective version. As a sample application, we prove that $PU_n^+=PO_n^+$.
\end{abstract}

\maketitle
\section*{Introduction}

Compact groups are often defined as matrix groups. That is, as a set of certain matrices. A typical example is the orthogonal group, which can be defined simply as the set of all orthogonal matrices
$$O_n=\{A\in M_n(\C)\mid A^i_j=\bar A^i_j,\; AA^t=1_{n\times n}=A^tA\}.$$

We can say essentially the same about Woronowicz compact quantum groups; the only difference is that we use a slightly different language. Quantum groups form a generalization of groups in non-commutative geometry. They arise by deforming the coordinate algebra of classical groups. We can show this on the example of the orthogonal group $O_n$. Instead of defining it as above, we can define the group in terms of the coordinate algebra $\Oalg(O_n)$, which is determined by generators and relations:
$$\Oalg(O_n)=\staralg(u^i_j,\;i,j=1,\dots,n\mid u^i_j=u^{i*}_j,\;uu^t=1_{n\times n}=u^tu, u^i_ju^k_l=u^k_lu^i_j).$$
Here, $u^i_j$ stand for the coordinate functions $u^i_j(A)=A^i_j$. The coordinate algebra obviously needs to be commutative (since complex functions should commute), but otherwise the definition exactly resembles the former one that works directly with elements of $O_n$. This is what the word \emph{presentation} in the title refers to.

Modifying this definition, we obtain several interesting \emph{quantum} groups. For example, the Wang's free orthogonal quantum group \cite{Wan95free}
$$\Oalg(O_n^+)=\staralg(u^i_j,\;i,j=1,\dots,n\mid u^i_j=u^{i*}_j,\;uu^t=1_{n\times n}=u^tu),$$
the \emph{anticommutative orthogonal quantum group} \cite{BBC07}
$$\Oalg(O_n^{-1})=\staralg\left(u^i_j\,\middle|\,\begin{matrix} u^i_j=u^{i*}_j,\;uu^t=1_{n\times n}=u^tu,\\u^i_ku^j_k=-u^j_ku^i_k,\;u^k_iu^k_j=-u^k_ju^k_i,\;u^i_ku^j_l=u^j_lu^i_k\end{matrix}\right),$$
or Banica's universal orthogonal quantum groups \cite{Ban96}
$$\Oalg(O^+(F)):=\staralg(u^i_j,\;i,j=1,\dots,n\mid\text{$u$ is unitary, }u=F\bar uF^{-1}).$$

%Determining groups and quantum groups this way appear of course naturally in many places. For instance, an automorphism group of a graph is a group of all permutations that preserve the graph structure. In the matrix language, it is the set of all permutation matrices that commute with the adjacency matrix. This has a corresponding analogue also in terms of \emph{quantum permutations}.

But there are also several groups and quantum groups that are defined in a different way; for instance, as quotient groups. For example, we define the \emph{projective orthogonal group} as the quotient of the orthogonal group $O_n$ by the global sign of the matrices $PO_n=O_n/\{\pm 1_{n\times n}\}$. We can define this group again as a matrix group as $PO_n=\{A\otimes A\mid A\in O_n\}$ since the matrix $A$ can be recovered from $A\otimes A$ up to a global sign. But we still may ask the following question: What is the set $\{A\otimes A\}$ exactly? What relations those matrices have to satisfy? In other words, what is the presentation of the algebra $\Oalg(PO_n)\subset\Oalg(O_n)$ using generators and relations? Answer to this question forms the main result of this article -- Theorem~\ref{T}.

\begin{thm*}
Consider a compact matrix quantum group with degree of reflection two $G\subset O^+(F)$, $F\in GL_n$. Consider an admissible generating set $S$ of $\RCat_G$ containing the duality morphism. Then the category $\RCat_{PG}$ is generated by $S':=\{T,\id_{\C^n}\otimes T\otimes\id_{\C^n}\mid T\in S\}$.
\end{thm*}

We use the language of representation categories and \emph{intertwiners} instead of \emph{relations} because it is much easier to formulate the general result. However, taking any particular instance, it is very easy to apply this to the concrete relations, which we illustrate on examples in Section~\ref{sec.examples}.

The motivation for such questions is the following type of problems: Suppose we need to prove that two quantum groups $G$ and $H$ are equal. If they are not defined by relations, then this problem might be very complicated. But if we are able to rewrite the definition of $\Oalg(G)$ and $\Oalg(H)$ using relations, then it is enough to compare those relations. We illustrate this approach in Section~\ref{sec.appl}, where we show that $PU_n^+=PO_n^+$ (more generally, $PU^+(F)=PO^+(F)$).

Note that a similar technique was used already in \cite[Section~5.2]{GroAbSym}, where the presentation of the group $PO_n$ and the quantum group $PO_n^{-1}$ was determined. But in that case, the $PO_n$ was represented not using the tensor product $u\otimes u$, but using the exterior product $u\wedge u$, which made it considerably more complicated. (The motivation there was again to show that certain two quantum groups are equal.) The current result is much simpler and more general.

%Finally, in Section~\ref{sec.moneq}, we continue studying the example of $PO_n^+$. In this case, there is an alternative way of obtaining the generators of the corresponding representation category based on the fact that $PO_n^+$ is monoidally equivalent with $S_{n^2}^+$ \cite{Ban02,RV10}. We study these representation categories by combinatorial means using partitions and provide a partition-based proof for the monoidal equivalence.

\section{Woronowicz--Tannaka--Krein}

In order to keep this note short, we skip the definition of compact matrix quantum groups and related notions. Apart from the original articles of Woronowicz \cite{Wor87,Wor88} we can recommend the survey \cite{Web17} or the author's PhD thesis \cite{GroThesis} for more details.

The main tool for formulating and proving our theorem is the so-called Woro\-nowicz--Tannaka--Krein duality \cite{Wor88}, which enables to express relations in the coordinate algebra in terms of the quantum group intertwiners.

For a compact matrix quantum group $G$ with fundamental representation $u$ of size $n$, we denote by
$$\RCat_G(k,l):=\Mor(u^{\otimes k},u^{\otimes l}):=\{T\colon (\C^n)^{\otimes k}\to (\C^n)^{\otimes l}\mid Tu^{\otimes k}=u^{\otimes l}T\}$$
the associated \emph{intertwiner spaces}. Those spaces form a \emph{concrete rigid monoidal $*$-category} $\RCat_G$, which we will refer to as the \emph{representation category} of $G$.

Given the coordinate Hopf $*$-algebra $\Oalg(G)$, we can also denote by $I_G\subset\C\langle x_{ij}\rangle$ the ideal of all relations that are satisfied in $\Oalg(G)$. That is $I_G=\{f\in\C\langle x_{ij}\rangle\mid f(u^i_j)=0\}$. Then the Woronowicz--Tannaka--Krein duality can be formulated as follows:

\begin{thm}[Woronowicz--Tannaka--Krein]
\label{T.TK}
Let $G\subset O^+(F)$ be a compact matrix quantum group with fundamental representation $u$. Then the relations in $\Oalg(G)$ are spanned by intertwiner relations:
\begin{equation}\label{eq.IG}
I_G=\spanlin\{Tu^{\otimes k}-u^{\otimes l}T\mid T\in\RCat_G(k,l)\}.
\end{equation}

Conversely, let $\RCat$ be any rigid monoidal $*$-category with $\N_0$ as the set of objects and $\RCat(k,l)\subset\Lin((\C^n)^{\otimes k},(\C^n)^{\otimes l})$ as the set of morphisms. Then the intertwiner relations, i.e.\ the formula \eqref{eq.IG} define a compact matrix quantum group $G$ with $\RCat_G=\RCat$. In this case, we have $G\subset O^+(F)$ with $F_i^j=R^{ij}$, where $R$ is the duality morphism in $\RCat_G$. So the $*$-operation on $\Oalg(G)$ has to be defined by $\bar u=F^{-1}uF$.
\end{thm}

Also important is the fact that the category operations in $\RCat_G$ exactly correspond to algebraic manipulations with the relations. Suppose that the category $\RCat_G$ is \emph{generated} by some set of intertwiners $S=\{T_1,\dots,T_m\}$. By this we mean that $\RCat_G$ is the smallest monoidal $*$-category with $\RCat_G(k,l)\subset\Lin((\C^n)^{\otimes k},(\C^n)^{\otimes l})$, which contains all the intertwiners $T_1,\dots,T_m$. We will denote this by $\RCat_G=\langle T_1,\dots,T_m\rangle_n$. Then it holds that the ideal $I_G$ is generated by the intertwiner relations in $S$ (and the unitarity for $u$). That is,
$$\Oalg(G)=\staralg(u^i_j\mid \bar u=F^{-1}uF,\; uu^*=1=u^*u,\; Tu^{\otimes k}=u^{\otimes l}T, T\in S(k,l)),$$
where the \emph{unitarity} relations $uu^*=1=u^*u$ can be equivalently replaced by \emph{orthogonality} relations $u(F^{-1}uF)^t=1=(F^{-1}uF)^tu$, which are exactly the relations corresponding to the duality morphisms $R$ and $R^*$.

Let us comment more in detail on why it is important to include the subscript $n$ in the notation $\langle\cdots\rangle_n$. In this article, we will identify the vector spaces $\C^n\otimes\C^n$ and $\C^{n^2}$. Consequently, a linear map $T\colon\C^{n^{2k}}\to\C^{n^{2l}}$ can be considered either as an intertwiner $(\C^n)^{\otimes 2k}\to(\C^n)^{\otimes 2l}$ or as an intertwiner $(\C^{n^2})^{\otimes k}\to(\C^{n^2})^{\otimes l}$. That is, it can either generate a category $\RCat=\langle T\rangle_n$, where an object $m\in\N_0$ is identified with $(\C^n)^{\otimes m}$, or it can generate a category $\RCat'=\langle T\rangle_{n^2}$, where an object $m\in\N_0$ is identified with $(\C^{n^2})^{\otimes m}$. These two categories might be quite different. The whole point of this article is to characterize the situation when the two categories are essentially the same ($\RCat'$ is a full subcategory of $\RCat$).

\begin{rem}[Diagrams]
We will sometimes use a diagrammatic description of the intertwiners here. In particular, the straight line $\idpart$ or more precisely the symbol $T_{\idpart}^{(n)}$ will denote the identity operator $\id\colon\C^n\to\C^n$. The \emph{pairing} $\pairpart$ or, more precisely, the symbol $T_{\pairpart}^{(n)}$ will denote the duality morphism $\C\to \C^n\otimes\C^n$. In the case $G\subset O_n^+$, we have $[T_{\pairpart}^{(n)}]^{ij}=\delta_{ij}$. From these two intertwiners, we can construct more complicated ones such as $\immersepart$ or $\Labba$ and so on. Note that all diagrams in this article are to be read top to bottom.
\end{rem}

\begin{rem}[Rotations]
Rotating the morphisms, sometimes referred to as the \emph{Frobenius reciprocity}, is a feature of rigid monoidal categories, which provides a bijective correspondence $\RCat(k,l)\leftrightarrow\RCat(k-1,l+1)$:
$$
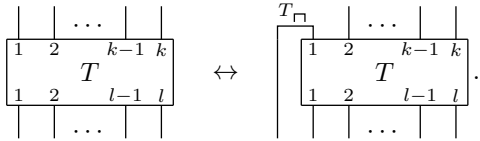
\BigPartition{
\Pline(0.5,1)(7.5,1)
\Pline(0.5,0)(7.5,0)
\Pline(0.5,0)(0.5,1)
\Pline(7.5,0)(7.5,1)
\Ptext(4,0.5){$T$}
\Pline(1,0)(1,-0.5)
\Pline(2.5,0)(2.5,-0.5)
\Ptext(4,-0.4){$\cdots$}
\Pline(5.5,0)(5.5,-0.5)
\Pline(7,0)(7,-0.5)
\Ptext(1,0.15){$\scriptstyle 1$}
\Ptext(2.5,0.15){$\scriptstyle 2$}
\Ptext(5.5,0.15){$\scriptstyle l-1$}
\Ptext(7,0.15){$\scriptstyle l$}
\Pline(1,1)(1,1.5)
\Pline(2.5,1)(2.5,1.5)
\Ptext(4,1.2){$\cdots$}
\Pline(5.5,1)(5.5,1.5)
\Pline(7,1)(7,1.5)
\Ptext(1,0.85){$\scriptstyle 1$}
\Ptext(2.5,0.85){$\scriptstyle 2$}
\Ptext(5.5,0.85){$\scriptstyle k-1$}
\Ptext(7,0.85){$\scriptstyle k$}
}
\quad\leftrightarrow\quad
\BigPartition{
\Pline(0.5,1)(7.5,1)
\Pline(0.5,0)(7.5,0)
\Pline(0.5,0)(0.5,1)
\Pline(7.5,0)(7.5,1)
\Ptext(4,0.5){$T$}
\Pline(1,0)(1,-0.5)
\Pline(2.5,0)(2.5,-0.5)
\Ptext(4,-0.4){$\cdots$}
\Pline(5.5,0)(5.5,-0.5)
\Pline(7,0)(7,-0.5)
\Ptext(1,0.15){$\scriptstyle 1$}
\Ptext(2.5,0.15){$\scriptstyle 2$}
\Ptext(5.5,0.15){$\scriptstyle l-1$}
\Ptext(7,0.15){$\scriptstyle l$}
\Pline(2.5,1)(2.5,1.5)
\Ptext(4,1.2){$\cdots$}
\Pline(5.5,1)(5.5,1.5)
\Pline(7,1)(7,1.5)
\Ptext(1,0.85){$\scriptstyle 1$}
\Ptext(2.5,0.85){$\scriptstyle 2$}
\Ptext(5.5,0.85){$\scriptstyle k-1$}
\Ptext(7,0.85){$\scriptstyle k$}
\Pblock 1to1.2:-0.5,1
\Pline (-0.5,1)(-0.5,-0.5)
\Ptext (0.2,1.4) {$\scriptstyle T_{\pairpart}$}
}.
$$

Consequently, any representation category is uniquely determined by the collection of spaces $\RCat(0,k)$, $k\in\N_0$ since all the morphism spaces $\RCat(k,l)$ can be recovered using rotations.

Note also that the spaces $\RCat(0,k)$ are closed under \emph{one-line rotations}:
$$
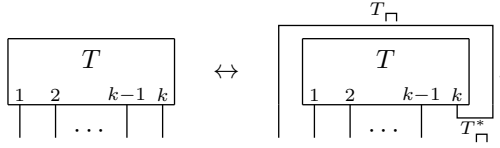
\BigPartition{
\Pline(0.5,1)(7.5,1)
\Pline(0.5,0)(7.5,0)
\Pline(0.5,0)(0.5,1)
\Pline(7.5,0)(7.5,1)
\Ptext(4,0.7){$T$}
\Pline(1,0)(1,-0.5)
\Pline(2.5,0)(2.5,-0.5)
\Ptext(4,-0.4){$\cdots$}
\Pline(5.5,0)(5.5,-0.5)
\Pline(7,0)(7,-0.5)
\Ptext(1,0.15){$\scriptstyle 1$}
\Ptext(2.5,0.15){$\scriptstyle 2$}
\Ptext(5.5,0.15){$\scriptstyle k-1$}
\Ptext(7,0.15){$\scriptstyle k$}
}
\quad\leftrightarrow\quad
\BigPartition{
\Pline(0.5,1)(7.5,1)
\Pline(0.5,0)(7.5,0)
\Pline(0.5,0)(0.5,1)
\Pline(7.5,0)(7.5,1)
\Ptext(4,0.7){$T$}
\Pline(1,0)(1,-0.5)
\Pline(2.5,0)(2.5,-0.5)
\Ptext(4,-0.4){$\cdots$}
\Pline(5.5,0)(5.5,-0.5)
\Ptext(1,0.15){$\scriptstyle 1$}
\Ptext(2.5,0.15){$\scriptstyle 2$}
\Ptext(5.5,0.15){$\scriptstyle k-1$}
\Ptext(7,0.15){$\scriptstyle k$}
\Pblock 0to-0.2:7,8.5
\Pblock 0to1.2:-0.5,8.5
\Pline (-0.5,0)(-0.5,-0.5)
\Ptext(7.8,-0.4){$\scriptstyle T_{\pairpart}^*$}
\Ptext(4,1.4){$\scriptstyle T_{\pairpart}$}
}
.
$$
\end{rem}

\section{Main result}
\label{sec.main}

The main question, which we are going to answer in this section is: If $G$ is a compact matrix quantum group defined by some relations, can we describe its projective version $PG$ in terms of relations as well?

\begin{defn}
Consider an arbitrary compact matrix quantum group $G\subset O^+(F)$ and denote by $u$ its fundamental representation. Then we define its \emph{projective version} $PG$ to be determined by the representation $v:=u\otimes u$ and coordinate algebra $\Oalg(PG)\subset\Oalg(G)$ generated by the entries $v^{ik}_{jl}=u^i_ju^k_l$.
\end{defn}

In the spirit of Tannaka--Krein duality, the questions about algebraic relations can equivalently be restated to a question about intertwiners. There is actually a straightforward way to describe the representation category of $PG$ in terms of the representation category of $G$. Namely, we obviously have that
$$\RCat_{PG}(k,l)=\Mor(v^{\otimes k},v^{\otimes l})=\Mor(u^{\otimes 2k},u^{\otimes 2l})=\RCat_G(2k,2l)$$
That is, we are just restricting the set of all objects $\N_0$ to the even numbers only. This corresponds to a well-known general fact that if we take a quotient quantum group, then its representation category is a full subcategory.

\begin{rem}\label{R.ogtwist}
Consider $G\subset O_n^+$. The representation category of $PG$ is, of course, a rigid monoidal $*$-category. But in this case, the duality morphism is not $T^{(n^2)}_{\pairpart}=T^{(n)}_{\Labab}$ since this is not an intertwiner of $G$ unless $G$ is a group. We need to use the duality morphism $T^{(n)}_{\Labba}$ instead.

Consequently, $PG$ is not a quantum subgroup of $O_{n^2}^+$, but only $PG\subset O^+(F)$ for a suitable $F$ (namely the flip map $\C^n\otimes\C^n\to\C^n\otimes\C^n$). In the language of relations, compare the standard orthogonality, which is of the form $\sum_k u^i_ku^j_k=\delta_{ij}=\sum_k u^k_iu^k_j$, with the somewhat twisted orthogonality, which holds here, $\sum_{k,l}v^{ab}_{kl}v^{cd}_{lk}=\delta_{ad}\delta_{bc}=\sum_{k,l}v^{kl}_{ab}v^{lk}_{cd}$. The entries are also not self-adjoint; instead, we have $v^{ij\,*}_{kl}=v^{ji}_{lk}$.
\end{rem}

However, the above observation does not help us regarding our question. It only says the obvious fact that $\Oalg(PG)$ satisfies exactly all the relations of $\Oalg(G)$ that make sense here. But we are rather interested in some small set of relations that would imply all the others. That is, we are looking for some generating set of the category.

Now things become a bit tricky since restricting the set of objects restricts the range of operations we can do. It can be well illustrated on the example of the rotation operation. Considering a linear map $T\in\C^{n^{2k}}$, we can do the rotation operation inside $\RCat:=\langle T_{\pairpart}^{(n)},T\rangle_n$
$$
R(T)=(\id_{\C^n}^{\otimes 2k}\otimes T_{\pairpart}^{(n)\,*})(\id_{\C^n}\otimes T\otimes\id_{\C^n})T^{(n)}_{\pairpart}=
\BigPartition{
\Pline(0.5,1)(7.5,1)
\Pline(0.5,0)(7.5,0)
\Pline(0.5,0)(0.5,1)
\Pline(7.5,0)(7.5,1)
\Ptext(4,0.7){$T$}
\Pline(1,0)(1,-0.5)
\Pline(2.5,0)(2.5,-0.5)
\Ptext(4,-0.4){$\cdots$}
\Pline(5.5,0)(5.5,-0.5)
\Ptext(1,0.15){$\scriptstyle 1$}
\Ptext(2.5,0.15){$\scriptstyle 2$}
\Ptext(5.5,0.15){$\scriptscriptstyle 2k-1$}
\Ptext(7,0.15){$\scriptscriptstyle 2k$}
\Pblock 0to-0.2:7,8.5
\Pblock 0to1.2:-0.5,8.5
\Pline (-0.5,0)(-0.5,-0.5)
\Ptext(7.8,-0.4){$\scriptstyle T_{\pairpart}^*$}
\Ptext(4,1.4){$\scriptstyle T_{\pairpart}$}
}
.
$$
Whereas for $\RCat':=\langle T_{\pairpart}^{(n)},T\rangle_{n^2}$, it is not guaranteed that this category is closed under such a rotation. Although $T_{\pairpart}$ is contained in $\RCat'$, we are not guaranteed that $\id_{\C^n}\otimes T\otimes\id_{\C^n}$ is contained in the category. This is because $\C^n$ is not an object in the category anymore, so we do not have the morphism $\id_{\C^n}$ at our disposal. We only can do the following kind of a rotation
$$
R'(T)=(\id_{\C^{n^2}}^{\otimes 2k}\otimes T_{\Labba}^{(n)\,*})(\id_{\C^{n^2}}\otimes T\otimes\id_{\C^{n^2}})T^{(n)}_{\Labba}=
\BigPartition{
\Pline(0.5,1)(7.5,1)
\Pline(0.5,0)(7.5,0)
\Pline(0.5,0)(0.5,1)
\Pline(7.5,0)(7.5,1)
\Ptext(4,0.7){$T$}
\Pline(1.5,0)(1.5,-0.5)
\Pline(2,0)(2,-0.5)
\Ptext(4,-0.4){$\cdots$}
\Ptext(1.75,0.15){$\scriptstyle 1$}
\Ptext(6.75,0.15){$\scriptscriptstyle k$}
\Pblock 0to-0.2:7,8
\Pblock 0to-0.4:6.5,8.5
\Pblock 0to1.4:-0.5,8.5
\Pblock 0to1.2:0,8
\Pline (-0.5,0)(-0.5,-0.5)
\Pline (0,0)(0,-0.5)
\Ptext(7.8,-0.6){$\scriptstyle T_{\Labba}^*$}
\Ptext(4,1.6){$\scriptstyle T_{\Labba}$}
}
.
$$
We essentially have $R'=R^2$.

To be more concrete, take $G=O_n^+$. The corresponding category is just generated by the duality morphism $T_{\pairpart}^{(n)}\in\Mor(1,u^{\otimes 2})=\RCat_G(0,2)$, so $\RCat_G=\langle T_{\pairpart}^{(n)}\rangle_n$. This is also an intertwiner of $PO_n^+$: $T_{\pairpart}^{(n)}\in\Mor(1,v)=\RCat_{PO_n^+}(0,1)$. But does $T_{\pairpart}^{(n)}$ generate the whole $\RCat_{PO_n^+}$? Does it hold that $\RCat_{PG}=\langle T_{\pairpart}^{(n)}\rangle_{n^2}=:\RCat'$? That is, are we able to construct all elements of $\RCat_{PG}$ using the category operations from $T_{\pairpart}^{(n)}$? This is not obvious and the answer turns out to be \emph{no}. Starting with $T_{\pairpart}^{(n)}\in\RCat'(0,1)\subset\RCat_{PG}(0,1)=\RCat_G(0,2)$, we can use the tensor product to obtain $T_{\Laabb}^{(n)}\in\RCat'(0,2)\subset\RCat_{PG}(0,2)=\RCat_G(0,4)$. Now inside $\RCat_G$, we can use the ``fine'' rotation to obtain $T_{\Labba}^{(n)}\in\RCat_G(0,4)=\RCat_{PG}(0,2)$, but we will never be able to do this insider $\RCat'=\langle T_{\pairpart}^{(n)}\rangle_{n^2}$.

Now let us formulate this as a formal mathematical statement. Before we do that, we need to handle one technical issue. We say that a quantum group $G\subset O^+(F)$ has \emph{degree of reflection two} if $\RCat_G(k,l)=\{0\}$ for every $k+l$ odd (see \cite{GWext,GroGlue} for more details). We will say that a generating set $S$ of such a category $\RCat$ is \emph{admissible} if it only contains elements of $\RCat_G(k,l)$, where both $k$ and $l$ are even. Note that any generating set $S$ of a category $\RCat_G$, where $G$ has degree of reflection two, can very easily be made admissible just by rotations (e.g. rotating every generator $T\in\RCat(k,l)$ to $\tilde T\in\RCat(0,k+l)$).

\begin{lem}
\label{L.PGCat}
Consider a compact matrix quantum group with degree of reflection two $G\subset O^+(F)$, $F\in GL_n$. Consider an admissible generating set $S$ of $\RCat_G$ containing the duality morphism. Denote $\RCat':=\langle S\rangle_{n^2}$. Then $\RCat'=\RCat_{PG}$ (that is, $S$ is a generating set of $\RCat_{PG}$) if and only if, for every $T\in\RCat'$, we have $\id_{\C^n}\otimes T\otimes\id_{\C^n}\in\RCat'$.
\end{lem}
\begin{proof}
The left-right implication is obvious: Taking $T\in\RCat_{PG}(k,l)=\RCat_G(2k,2l)$, we must have $\id_{\C^n}\otimes T\otimes\id_{\C^n}\in\RCat_{PG}(k+1,l+1)=\RCat_G(2k+2,2l+2)$ since $\RCat_G$ is closed under tensor products and contains $\id_{\C^n}\in\RCat_G(1,1)$.

Now, we prove the right-left implication. Denote by $T_{\pairpart}$ the duality morphism in $\RCat_G$ (corresponding to the object 1). Since we assume that $T_{\pairpart}\in S$, we must in particular have $T_{\pairpart}\in\RCat(0,1)$ (which does not play the role of the duality morphism in $\RCat'$ since it intertwines the object 0 with the object 1, not with 2). Now, we can define $T_{\Labba}:=(\id_{\C^n}\otimes T_{\pairpart}\otimes\id_{\C^n})T_{\pairpart}$, which can now play the role of the duality morphism in $\RCat'$. Therefore, $\RCat'$ is a rigid monoidal $*$-category. In particular, we can perform rotations in $\RCat'$ (but we always have to remember to rotate the whole object $1\simeq\C^{n^2}$).

Consequently, we can now use the Frobenius reciprocity. This means, that it is enough now to prove that $\RCat'(0,k)=\RCat_{PG}(0,k)=\RCat_G(0,2k)$ for every $k\in\N_0$. Notice that the spaces $\RCat'(0,k)$ are closed under the operation $R'^{1/2}\colon\RCat'(0,k)\to\RCat'(0,k)$ defined by
$$R'^{1/2}(T)=(\id_{\C^n}^{\otimes 2k}\otimes T_{\pairpart}^*)(\id_{\C^n}\otimes T\otimes\id_{\C^n})T_{\pairpart}=
\BigPartition{
\Pline(0.5,1)(9.5,1)
\Pline(0.5,0)(9.5,0)
\Pline(0.5,0)(0.5,1)
\Pline(9.5,0)(9.5,1)
\Ptext(5,0.7){$T$}
\Pline(1,-0.2)(1,-0.6)
\Pline(2,-0.2)(2,-0.6)
\Pline(3,-0.2)(3,-0.6)
\Pline(4,-0.2)(4,-0.6)
\Ptext(5,-0.4){$\cdots$}
\Pline(6,-0.2)(6,-0.6)
\Pline(7,-0.2)(7,-0.6)
\Pline(8,-0.2)(8,-0.6)
\Pline(1,-0.2)(1.4,0)
\Pline(2,-0.2)(1.6,0)
\Pline(3,-0.2)(3.4,0)
\Pline(4,-0.2)(3.6,0)
\Pline(6,-0.2)(6.4,0)
\Pline(7,-0.2)(6.6,0)
\Pline(8,-0.2)(8.4,0)
\Ptext(1.5,0.15){$\scriptstyle 1$}
\Ptext(3.5,0.15){$\scriptstyle 2$}
\Ptext(6.5,0.15){$\scriptstyle k-1$}
\Ptext(8.5,0.15){$\scriptstyle k$}
\Pblock 0to-0.2:8.6,10
\Pblock -0.2to1.2:0,10
\Pline(0,-0.2)(0,-0.6)
\Ptext(9.5,-0.4){$\scriptstyle T_{\pairpart}^*$}
\Ptext(5,1.4){$\scriptstyle T_{\pairpart}$}
},$$
which somehow mimics the rotation $R\colon\RCat_G(0,k)\to\RCat_G(0,k)$. We claim that this is already enough to finish the proof. Here is the reason why:

We define the collection of linear spaces $\RCat(k,l)$ by $\RCat(0,2k)=\RCat'(0,k)$, $\RCat(0,\ppen 2k+1)=\{0\}$ and finally $\RCat(k,l)$ is defined by rotations using $T_{\pairpart}$. If we prove that $\RCat$ forms a category, then it must be the category generated by $S$, which is exactly $\RCat_G$, so this will prove the equality $\RCat'(0,k)=\RCat(0,2k)=\RCat_G(0,2k)=\RCat_{PG}(0,k)$. The question whether $\RCat$ forms a category can be characterized purely in terms of the spaces $\RCat(0,k)$ \cite[Prop.~3.11]{GroGlue} -- namely it has to be closed under tensor products, contractions, one-line rotations, and reflections (see \cite[Def.~3.10]{GroGlue} for details). The only problem here is the rotation since the standard rotation in $\RCat'$ induces rotation by two objects in $\RCat$. However, the required ``finer'' rotation by a single object is induced by the ``half-rotation'' $R'^{1/2}$, which we introduced above. Consequently, $\RCat$ is indeed a category, which is what we wanted to show.
\end{proof}

\begin{thm}
\label{T}
Consider a compact matrix quantum group with degree of reflection two $G\subset O^+(F)$, $F\in GL_n$. Consider an admissible generating set $S$ of $\RCat_G$ containing the duality morphism. Then the category $\RCat_{PG}$ is generated by $S':=\{T,\id_{\C^n}\otimes T\otimes\id_{\C^n}\mid T\in S\}$.
\end{thm}
\begin{proof}
Denote by $\RCat=\langle S'\rangle_{n^2}$ the category generated by $S'$. To show that $\RCat=\RCat_{PG}$, we will use the preceding lemma. That is, we need to show that $\id_{\C^n}\otimes T\otimes\id_{\C^n}\in\RCat$ for every $T\in\RCat$. We prove this by induction on the number of category operations it is needed to generate $T$ from the corresponding generators in $S'$.

Effectively, we need to do the following. Take any $T_1,T_2\in\RCat$ such that also $\id_{\C^n}\otimes T_j\otimes\id_{\C^n}\in\RCat$, $j=1,2$. Then we need to show that all of the following $\id_{\C^n}\otimes T_1\otimes T_2\otimes\id_{\C^n}$, $\id_{\C^n}\otimes (T_2T_1)\otimes\id_{\C^n}$ (if $T_1$ and $T_2$ are composable), and $\id_{\C^n}\otimes T^*\otimes\id_{\C^n}$ are contained in $\RCat$.

The cases of composition and involution are obvious, so let us show it for the case of the tensor product. So, obviously, we can construct the following element $\id_{\C^n}\otimes T_1\otimes\id_{\C^n}\otimes\id_{\C^n}\otimes T_2\otimes\id_{\C^n}$ (simply using the tensor product in $\RCat$). To get rid of the identities in the middle, we need to sandwich this with elements of the form $\id_{\C^n}^{\otimes 2k_1}\otimes \id_{\C^n}\otimes T_{\pairpart}\otimes \id_{\C^n}\otimes \id_{\C^n}^{\otimes 2k_2}$, where $k_1,k_2\in\N_0$ and $T_{\pairpart}$ is the duality morphism of $\RCat_G$. Such an element must be in $\RCat$ since the even powers of the identity are there by definition and we have $T_{\pairpart}\in S$, so $\id_{\C^n}\otimes T_{\pairpart}\otimes\id_{\C^n}\in S'\subset\RCat$.
\end{proof}

\section{Examples}
\label{sec.examples}

\begin{ex}[Free orthogonal quantum group]
\label{E.free}
As the first example, take the universal orthogonal quantum group $O^+(F)$, whose representation category is generated just by the duality morphism $T^{(n)}_{\pairpart}$ \cite{Ban96}. Applying Theorem~\ref{T}, we get that the representation category of $PO^+(F)$ is generated by the set $\{T^{(n)}_{\pairpart},T^{(n)}_{\immersepart}\}$. Equivalently, we can use the generating set $\{T^{(n)}_{\Labba},T^{(n)}_{\immersepart}\}$, which contains the duality morphism $T^{(n)}_{\Labba}$.

Applying Tannaka--Krein duality, we can translate this result into algebraic relations. To simplify the formulas, take $G:=O_n^+$ the ordinary free orthogonal quantum group. Here, we have $[T_{\pairpart}^{(n)}]^{ij}=\delta_{ij}$. The free orthogonal quantum group can be defined as $O_n^+=(\Oalg(O_n^+),u)$, where $\Oalg(O_n^+)$ is the universal algebra generated by the elements $u^i_j$, $i,j=1,\dots,n$ that are self-adjoint and satisfy the orthogonality relation $uu^t=1=u^tu$, i.e.\ $\sum_k u^i_ku^j_k=\delta_{ij}=\sum_k u^k_iu^k_j$.

Writing down the relations for $2n\times 2n$ matrix $v$ associated to the intertwiners $T^{(n)}_{\pairpart}$ and $T^{(n)}_{\immersepart}$, we get a presentation of $PO_n^+$. Namely, we get that $PO_n^+=(\Oalg(PO_n^+),v)$ with
$$\Oalg(PO_n^+)=\staralg\left(v^{ik}_{jl}\,\middle|\,\begin{matrix}\sum_{k,l}v^{ab}_{kl}v^{cd}_{lk}=\delta_{ad}\delta_{bc}=\sum_{k,l}v^{kl}_{ab}v^{lk}_{cd},\\v^{ik\,*}_{jl}=v^{ki}_{lj},\;\sum_kv^{ai}_{bk}v^{jc}_{kd}=\delta_{ij}v^{ac}_{bd}\end{matrix}\right).$$

The relation $v^{ik\,*}_{jl}=v^{ki}_{lj}$ is essentially just a definition of the involution on $\Oalg(PO_n^+)$, which comes (by Theorem~\ref{T.TK}) from the fact that $PG\subset O^+(F')$, where $F'$ is given by the new duality morphism $[F']_{ij}^{kl}=[T_{\Labba}^{(n)}]^{ijkl}=\delta_{il}\delta_{jk}$. The first relation is the orthogonality relation corresponding to the duality morphisms $T^{(n)}_{\Labba}$ and $T^{(n)}_{\Uabba}$ (cf.\ Remark~\ref{R.ogtwist}). Finally, the third relation corresponds to the second generator $T_{\immersepart}^{(n)}$. 

For the general case $O^+(F)$, we get $PO^+(F)\subset O^+(F')$, where $[F']^{kl}_{ij}=F_i^lF_j^k$ and
$$\Oalg(PO^+(F))=\staralg(v^{ik}_{jl}\mid \bar v=F'^{-1}vF',\;vv^*=1=v^*v,\;\sum_{k,l}v^{ai}_{bk}v^{jc}_{ld}F_k^l=F_i^jv^{ac}_{bd}).$$
\end{ex}

\begin{rem}
\label{R.moneq}
The generating set for $\RCat_{PO_n^+}$ might not be too surprising for experts. It is known that $PO_n^+$ is monoidally equivalent with $S_{n^2}^+$. The representation category of $S_{n^2}^+$ may be modelled by non-crossing partitions and then the monoidal equivalence is given by the ``fattening'' procedure
$${\def\Pwidth{0.2}\LPartition{0.5:2}{0.5:3,4,5;0.9:1,6}}\qquad\leftrightarrow\qquad\LPartition{}{1:0.9,6.2;0.8:1.2,5.9;0.6:2.9,5.2;0.4:3.2,3.9;0.4:4.2,4.9;0.5:1.9,2.2}.$$
It is also known that the category of all partitions is generated by ${\def\Pwidth{0.2}\La}$ and ${\def\Pwidth{0.2}\forkpart}$. Applying the procedure above, we get the generators $\LPartition{}{0.5:0.9,1.1}$ and
$
\Partition{
\Pblock 0to0.40:1.1,1.9
\Pline (0.9,0)(0.9,0.55)
\Pline (0.9,0.55)(1.4,0.55)
\Pline (1.4,0.55)(1.4,1)
\Pline (2.1,0)(2.1,0.55)
\Pline (2.1,0.55)(1.6,0.55)
\Pline (1.6,0.55)(1.6,1)
}
$ for $\RCat_{PO_n^+}$.

(While $S_{n^2}^+$ is the quantum automorphism group of the classical set of $n^2$ points \cite{Wan98}, the quantum group $PO_n^+$ is the quantum automorphism group of $M_n(\C)$ \cite{Ban99}. The fact that all quantum symmetry groups of so-called \emph{finite quantum spaces} of a given size are described by the same representation category is implicitly contained already in \cite{Ban02}. It was formulated as an explicit statement in \cite[Theorem~4.7]{RV10}. The diagrammatic isomorphism was described in detail in \cite[Section~4]{KS08}.)
\end{rem}

\begin{ex}[Orthogonal group]
We can do the same with the orthogonal group $O_n$. Its representation category is generated by $\{T^{(n)}_{\pairpart},T^{(n)}_{\crosspart}\}$. So, the representation category of $PO_n$ must be generated by $\{T^{(n)}_{\pairpart},T^{(n)}_{\immersepart},T^{(n)}_{\crosspart},T^{(n)}_{\IcrosspartI}\}$. The generator $T_{\pairpart}^{(n)}$ can again be replaced by $T^{(n)}_{\Labba}$ and the last one $T^{(n)}_{\IcrosspartI}$ can actually be omitted as we have
$$
\IcrosspartI=
\Partition{
\Pline (1,2) (1,1.1)
\Pblock 1.1to1.5:2,3
\Pline (4,2) (4,1.1)
\Pline (5,2) (5,1.1)
\Pline (6,2) (6,1.1)
\Pline (1,1) (1,0)
\Pline (2,1) (2,0)
\Pline (3,1) (4,0)
\Pline (4,1) (3,0)
\Pline (5,1) (5,0)
\Pline (6,1) (6,0)
\Pline (1,-1) (1,-0.1)
\Pline (2,-1) (2,-0.1)
\Pline (3,-1) (3,-0.1)
\Pblock -0.1to-0.5:4,5
\Pline (6,-1) (6,-0.1)
}
$$

Hence, we can write
$$\Oalg(PO_n^+)=\staralg\left(v^{ik}_{jl}\,\middle|\,\begin{matrix}\sum_{k,l}v^{ab}_{kl}v^{cd}_{lk}=\delta_{ad}\delta_{bc}=\sum_{k,l}v^{kl}_{ab}v^{lk}_{cd},\\v^{ik\,*}_{jl}=v^{ki}_{lj}=v^{ik}_{jl},\;\sum_kv^{ai}_{bk}v^{jc}_{kd}=\delta_{ij}v^{ac}_{bd}\end{matrix}\right).$$

$PO_n$ is of course a group, so the corresponding algebra $\Oalg(PO_n)$ should be commutative. This is not so obvious looking at the relations only, but it is quite easy to derive using diagrams. Let us indicate the relation-based proof:
$$v^{ik}_{jl}v^{ac}_{bd}=\sum_{r,s}v^{ir}_{jb}v^{rk}_{sl}v^{ac}_{sd}=\sum_{r,s}v^{ir}_{jb}v^{kr}_{ls}v^{ac}_{sd}=v^{ia}_{jb}v^{kc}_{ld}=v^{ai}_{bj}v^{ck}_{dl}=\cdots=v^{ac}_{bd}v^{ik}_{jl}.$$

Since $PO_n$ is a group, we can also rewrite this in the standard matrix group notation. If we define $PO_n:=\{A\otimes A\mid A\in O_n\}$, then Theorem~\ref{T} implies that
$$PO_n=\left\{B\in M_{n^2}(\R)\,\middle|\,\begin{matrix}\sum_{k,l}B^{ab}_{kl}B^{cd}_{lk}=\delta_{ad}\delta_{bc}=\sum_{k,l}B^{kl}_{ab}B^{lk}_{cd},\\B^{ik\,*}_{jl}=B^{ki}_{lj}=B^{ik}_{jl},\;\sum_kB^{ai}_{bk}B^{jc}_{kd}=\delta_{ij}B^{ac}_{bd}\end{matrix}\right\}.$$
\end{ex}

\begin{ex}[Anticommutative orthogonal quantum group]
Applying the anticommutative twist to the preceding example, we get the same results for the anticommutative orthogonal quantum group $O_n^{-1}$. Recall from \cite{BBC07,GWgen} that the representation category of $O_n^{-1}$ is monoidally isomorphic to the one of $O_n$. To be concrete, it is generated by $\{\breve T^{(n)}_{\pairpart},\breve T^{(n)}_{\crosspart}\}$, with $[\breve T^{(n)}_{\pairpart}]^{ij}=\delta_{ij}$ and $[\breve T^{(n)}_{\crosspart}]^{ij}_{kl}=-\delta_{il}\delta_{jk}+2\delta_{ijkl}$.

Consequently, we can essentially copy all the results from the previous example. In particular, we have that the category $\RCat_{PO_n^{-1}}$ is generated by $\{\breve T^{(n)}_{\pairpart},\breve T^{(n)}_{\immersepart},\breve T^{(n)}_{\crosspart}\}$.
\end{ex}

\section{Application: $PU_n^+=PO_n^+$}
\label{sec.appl}

In this section, we would like to illustrate how the results can be further applied. We use the concrete presentation of $PO^+(F)$ by relations to show that it is equal to $PU^+(F)$.

The \emph{universal unitary quantum group} $U^+(F)$ is defined through the following algebra \cite{Ban97}
$$\Oalg(U^+(F))=\staralg(u^i_j\mid \text{$u$ and $F\bar uF^{-1}$ are unitary}).$$

In particular, choosing $F=\id$, we get the \emph{free unitary quantum group} $U_n^+$ defined earlier in \cite{Wan95free}
$$\Oalg(U^+(F))=\staralg(u^i_j\mid uu^*=1_{n\times n}=u^*u,\; \bar uu^t=1_{n\times n}=u^t\bar u).$$

For any $G\subset U^+(F)$, we define its projective version $PG$ to be the quantum group with fundamental representation $v=u\otimes F\bar uF^{-1}$ and coordinate algebra $\Oalg(PG)\subset\Oalg(G)$ generated by the entries of $v$. Note that this is consistent with the definition for $G\subset O^+(F)$ since in that case we have $u=F\bar uF^{-1}$.

\begin{thm}
It holds that $PU^+(F)=PO^+(F)$. In particular, $PU_n^+=PO_n^+$.
\end{thm}
\begin{proof}
Since $O^+(F)\subset U^+(F)$ (by which we mean that $\Oalg(O^+(F))$ is a quotient of $\Oalg(U^+(F))$), it is obvious that $PO^+(F)\subset PU^+(F)$. On the other hand, it is straightforward to check that $PU^+(F)$ satisfies the relations that we derived for $PO^+(F)$, which proves the opposite inclusion $PU^+(F)\subset PO^+(F)$.

Let us do it here explicitly for the case $PU_n^+=PO_n^+$. The relations of $PO_n^+$ were explicitly derived in Example~\ref{E.free}. Denoting by $v=u\otimes\bar u$ the fundamental representation of $U_n^+$, we indeed have
\begin{align*}
v^{ik\,*}_{jl}&=(u^i_ju^{k\,*}_l)^*=u^k_lu^{i\,*}_j=v^{ki}_{lj},\\
vv^*&=(uu^*\otimes\bar uu^t)=1=(u^*u\otimes u^t\bar u)=v^*v,\\
\sum_kv^{ai}_{bk}v^{jc}_{kd}&=\sum_ku^a_bu^{i\,*}_ku^j_ku^{c\,*}_d=\delta_{ij}v^{ac}_{bd}.\qedhere
\end{align*}
\end{proof}

\begin{rem}
There is actually also an alternative proof for this statement. In \cite[Théorème~1(iv)]{Ban97}, it was proven that $U^+(F)$ is a \emph{free complexification} of $O^+(F)$, denoted sometimes as $U^+(F)=O^+(F)\mathbin{\tilde *}\hat\Z$ (see also \cite{GroGlue}). This means that $\Oalg(U^+(F))\subset\Oalg(O^+(F))*\C\Z$ such that $u^i_j=w^i_jz$, where $w$ is the fundamental representation of $O^+(F)$ and $z$ is the generator of $\C\Z$. Consequently, it is easy to see that $v^{ik}_{jl}=w^i_jzz^*w^{k*}_l=w^i_jw^{k*}_l$, so it equals to the fundamental representation of $PO^+(F)$.
\end{rem}

\bibliographystyle{halpha}
\bibliography{mybase}

\end{document}